%% file: md-stacks.tex
\documentclass[a4paper,11pt,reqno]{amsart}

\input{header-md-stacks}

\title{Mori dream stacks}

\begin{document}

\maketitle

\begin{abstract}
We propose a generalisation of Mori dream spaces to stacks.
We show that this notion is preserved under root constructions and taking abelian gerbes.
Unlike the case of Mori dream spaces, such a stack is not always given as a quotient of the spectrum of its Cox ring by the Picard group. We give a criterion when this is true in terms of Mori dream spaces and root constructions. Finally, we compare this notion with that of smooth toric Deligne-Mumford stacks.
\end{abstract}

\tableofcontents 

\section*{Introduction}
\addtocontents{toc}{\protect\setcounter{tocdepth}{1}}

In \cite{Hu-Keel}, Yi Hu and Se\'an Keel introduce the notion of a Mori dream space, which is a projective variety satisfying a number of properties ensuring that the Minimal Model Program works nicely. 
Moreover, they characterise Mori dream spaces as exactly those projective varieties, which can be written as a GIT~quotient of the spectrum of its \emph{Cox ring} $\RR(X)$ by $\Hom(\Cl(X),\CC^*)$.
We want to note here, that the unstable locus of $\Spec \RR(X)$ is 
the so-called \emph{irrelevant locus} $V(J_{\irr})$, where all global sections of an ample line bundle on $X$ vanish.
In [loc.~sit.], $\Cl(X)$ is assumed to be free of finite rank.
In \cite{Hausen}, J\"urgen Hausen extends the definition to include the case where $\Cl(X)$ is just finitely generated.

In this paper, we propose a definition of Mori dream stacks as a generalisation of Mori dream spaces. For us, a \emph{Mori dream stack} $\XX$ (or \MD-stack for short) is a smooth Deligne-Mumford stack with only constant invertible functions, finitely generated Picard group $\Pic(\XX)$ and finitely generated Cox ring $\RR(\XX)$. 
It turns out that an \MD-stack $\XX$ is not automatically of the form
\[ \XX = \stackquot{\Spec \RR(\XX) \setminus V(J_{\irr}) }{ \Hom(\Pic(\XX),\kk^*) },  \]
as in the case of varieties. If $\XX$ is such a quotient, we call it an \emph{\MD-quotient stack}.
Without the smoothness of $\XX$, most of our results on \MD-quotient stacks will not hold; e.g., see Example~\ref{ex:dicyclic}. Moreover, we will also need that the coarse moduli space of $\XX$ has at most quotient singularities.
  
We prove that \MD-stacks are preserved by classical constructions for stacks, like the root constructions with simple normal crossing divisors and by taking abelian gerbes; see Theorems \ref{theorem.divisor root preserves MD} and \ref{thm.gerbes preserve MD}. \MD-quotient stacks are also preserved by these root constructions. However, they are not preserved by taking arbitrary abelian gerbes, only by those which come from taking roots of line bundles.

The main result of this paper is the following characterisation of \MD-stacks and of the subclass of \MD-quotient stacks. Let $\XX$ be a Deligne-Mumford stack and $\XX \to X$ be the structure morphism to its \emph{coarse moduli space} $X$. This morphism factors through $\XX \to \Xcan \to X$, where $\Xcan$ is the \emph{canonical stack} i.e., the minimal orbifold having $X$ as its coarse moduli space. 
In the following see Definition~\ref{def:mdspace} for the notion of an \MD-space.

\begin{maintheorem}[Theorems \ref{thm.moduli space is MD}, \ref{thm:main_thm} and Corollary \ref{cor.all stacks with a MD moduli space are MD}]
Let $\XX$ be a smooth Deligne-Mumford stack and $\XX \to X$ be the structure morphism to its coarse moduli space $X$. 
\begin{enumerate}
\item If $\XX$ is an \MD-stack, then $X$ is an \MD-space. Moreover, if $\XX$ has abelian generic stabiliser and a ramification divisor with simple normal crossing components, then the converse holds too.
\item $\XX$ is an \MD-quotient stack
if and only if $X$ is an \MD-space with the property that $\Spec \RR(X) \setminus V(J_{\irr})$ is smooth and $\XX$ can be obtained from $\Xcan$ by root constructions with simple normal crossing divisors and line bundles.
\end{enumerate}
\end{maintheorem}   

The easiest class of Mori dream spaces are toric varieties. They are characterised among the Mori dream spaces by the property that their Cox rings are polynomial rings. Moreover a Mori dream space can be embedded as a closed subvariety in a toric variety in such a way that their Picard groups are naturally isomorphic; see \cite{Hu-Keel}.
In \cite{Borisov-etal}, Lev Borisov, Linda Chen and Greg Smith introduce toric stacks. 
In the later work \cite{Fantechi-etal}, Barbara Fantechi, \'Etienne Mann and Fabio Nironi give a geometric definition of a smooth toric
Deligne-Mumford stack with an action of a Deligne-Mumford torus. 
They show that any such stack can be obtained from its coarse moduli space, which is again toric, by root constructions along torus-invariant divisors and line bundles.
Many of our results are a generalisation of their ideas to the \MD-setting.

The last part of our article is dedicated to the relation between \MD-quotient stacks and smooth toric Deligne-Mumford stacks. As a generalisation of \cite{Hu-Keel} mentioned above, we show the following theorem.
\begin{theorem*}[Theorems~\ref{thm. MDquotient Cox ring polyn is toric} and \ref{thm. MDquotient embedded in toric}]
Let $\XX$ be an \MD-quotient stack. Then:
\begin{enumerate}
\item $\XX$ is a smooth toric Deligne-Mumford stack if and only if its Cox ring is a polynomial ring.
\item If the divisor class group of the coarse moduli space $X$ of $\XX$ is free, then there is a closed embedding $\iota\colon\XX \into \YY$ into a smooth toric Deligne-Mumford stack $\YY$ such that $\iota^*\colon \Pic(\YY) \xrightarrow{\sim} \Pic(\XX)$ is an isomorphism.
\end{enumerate}
\end{theorem*}

\subsection*{Notation and Conventions}
In this article, $\kk$ will always be an algebraically closed field of characteristic $0$,
since most results on Mori dream spaces and GIT quotients are only available in this setting.
A stack is always meant to be algebraic, separated and of finite type over $\kk$.
We work in the \'etale topology.

\subsection*{Acknowledgements}

The authors would like to thank Nicola Pagani and Angelo Vistoli for answering our questions; Barbara Fantechi, Flavia Poma and Fabio Tonini for very helpful discussions and explanations; Cinzia Casa\-grande for stimulating questions; and again Fabio Tonini for reading carefully a preliminary version of this article.
Finally, we want to thank the anonymous referee for helpful comments.

\addtocontents{toc}{\protect\setcounter{tocdepth}{2}}

\section{Preliminaries}

\subsection{Mori dream spaces and Cox rings}
\label{sec:md_spaces}

Our definition of a Mori dream space will differ from the standard terminology coming from the Minimal Model Program. To avoid confusion, we will always use the abbreviation \MD-space when referring to the definition given below.
The reason is that we rely on a quotient description of such a variety, which comes in handy when passing to stacks.
Therefore, our definition 
will be based on the characterisation of projective, $\QQ$-factorial Mori dream spaces as GIT quotients as given in \cite[Prop.~2.9]{Hu-Keel},
and the more general one of a variety which is a GIT quotient of the spectrum of its (finitely generated) Cox ring as found in~\cite{Coxrings}.

\begin{definition}
\label{def:mdspace}
A normal variety $X$ will be called an \emph{\MD-space} if
\begin{enumerate}
\item $X$ has at most quotient singularities,
\item $H^0(X,\OO_X^*) = \kk^*$,
\item $\Cl(X)$ is finitely generated as a $\ZZ$-module,
\item its Cox ring $\RR(X)$ is finitely generated as a $\kk$-algebra.
\end{enumerate}
\end{definition}

\begin{remark}
Note that a variety has by definition \emph{quotient singularities} if there exists an analytical open cover by sets of the form $U/G$, where $U$ is smooth and $G$ a finite group. Consequently, such a variety is $\QQ$-factorial.
Moreover, this condition will allow us to associate to $X$ a smooth stack in a canonical way.

If $X$ in the definition is moreover projective, then our definition is a slightly more restrictive than the standard one, in which $\QQ$-factorial singularities are allowed; see~\cite[Prop.~2.9]{Hu-Keel}.
\end{remark}

Let $X$ be an \MD-space. For us the most important feature is that $X$ has a description as a GIT quotient of the spectrum of its Cox ring $\RR(X)$ by $\Hom(\Cl(X),\kk^*)$.
Let us introduce all the ingredients necessary to understand this quotient description. 
The Cox ring can be na\"ively defined as the $\Cl(X)$-graded vector space
\[
\RR(X) = \bigoplus_{L \in \Cl(X)} H^0(X,L),
\]
but as pointed out in \cite{Hu-Keel}, this definition does not admit a canonical ring structure. To obtain one in the case when $\Cl(X)$ is free, one has to fix a basis of $\Cl(X)$. Then the product can be defined by multiplying the sections as rational functions.

For an arbitrary finitely generated $\Cl(X)$, the definition of a ring structure is a bit more technical; see \cite{Hausen}. 
In the following presentation we will follow \cite[Sec.~I.4]{Coxrings}. 
 
Choose a finite number of divisors $D_1,\ldots,D_m$, such that the classes $[D_1],\ldots,[D_m]$ generate $\Cl(X)$. Denote by $K$ the free subgroup of the group of Weil divisiors  
$\WDiv(X)$ generated by these divisors. 
We hence have a surjection $K \onto \Cl(X)$. 
If we write $K^0$ for its kernel, then we obtain a short exact sequence  $K^0 \into K \onto \Cl(X)$ with $K,K^0$ free groups. 
Finally, since the composition of $K^0 \into K \into \WDiv(X)$ and $\WDiv(X) \onto \Cl(X)$ is zero, the map $K^0 \into \WDiv(X)$ has to factor uniquely through $\kk(X)^*$, that is, the kernel of $\WDiv(X)\onto \Cl(X)$. So we obtain a map $\chi\colon K^0 \into \kk(X)^*$. In particular, for a principal divisor $E \in K^0$ we have $E = \div \chi(E)$. 
 
Consider the sheaf of $\Cl(X)$-graded rings $\mathcal{S} = \bigoplus_{D \in K} \OO_X(D)$ containing the homogeneous ideal $\mathcal{I} = \genby{1 - \chi(E) \mid E \in K^0}$. 
 
\begin{definition} 
The \emph{Cox ring} of $X$ is the $\Cl(X)$-graded ring 
\[ 
\RR(X) \coloneqq H^0(X,\mathcal{S}) / H^0(X,\mathcal{I}). 
\]  
\end{definition}

We note that the definition of the Cox ring depends on the choices we made above. But for different choices, the resulting rings will be (non-canonically) isomorphic.

The $\Cl(X)$-grading of the Cox ring $\RR(X)$ is equivalent to an action of $\Hom(\Cl(X),\kk^*)$ on $\Spec \RR(X)$. 
Actually, $X$ arises as a GIT quotient by this action in the following way.

\begin{proposition}[{\cite[Prop.~2.9]{Hu-Keel}, \cite[Sec.~I.6]{Coxrings}}] \label{prop.MD space as quotient}
Let $X$ be an \MD-space and $L$ an ample line bundle on $X$, giving a character $\chi\colon \Hom(\Cl(X),\kk^*) \to \kk^*$.
Then $X$ is isomorphic to the GIT quotient $\Spec \RR(X) \git_\chi \Hom(\Cl(X),\kk^*)$.
More explicitly, 
define the \emph{irrelevant ideal} $J_{\irr} = J_{\irr}(X) \coloneqq \sqrt{\genby{H^0(\XX,L)}}$, that turns out to be independent of the choice of the ample line bundle $L$.
Then we obtain
\[
X = \quot{\Spec \RR(X) \setminus V(J_{\irr})}{\Hom(\Cl(X),\kk^*)}.
\]
\end{proposition}

\subsection{Stacks}
We recall that we assume stacks to be algebraic, separated and of finite type over $\kk$.
Most of the constructions of this section will hold also if the stacks are not Deligne-Mumford.   
In this section we recall some basic facts on stacks, that will be useful for the following. 

\smallskip

The first notion we need is that of a coarse moduli space. The main references for us are \cite{Laumon-etal}, \cite{Olsson} and \cite{Vistoli}.   

\begin{definition}
 Let $\XX$ be a stack. A \emph{coarse moduli space} for $\XX$ is an algebraic space $X$ over $\kk$ with a morphism $\pi\colon \XX \to X$, such that:
 \begin{enumerate}
  \item the following universal property holds: for every morphism $\psi\colon \XX \to Z$ to an algebraic space, there exists a unique morphism $f\colon X \to Z$, 
  such that $\psi= f \circ \pi$; 
  \item \label{it:iso_pts} the map $|\XX(\kk)| \to X(\kk)$ is bijective, where $|\XX(\kk)|$ denotes the set of isomorphism classes in $\XX(\kk)$.
 \end{enumerate}
\end{definition}

The following theorem ensures the existence of coarse moduli spaces.

\begin{proposition} [{\cite{Keel-Mori}, \cite[Thm.~1.1]{Conrad}}]  
\label{prop:keel_mori}
Let $\XX$ be a (separated) Deligne-Mumford stack.
Then there exists a coarse moduli space $\pi \colon \XX \to X$ which is proper and quasi-finite.
Moreover, if $X'\to X$ is a flat map of algebraic spaces, then $\XX \times_X X' \to X'$ is a coarse moduli space, too.
\end{proposition} 

The next proposition actually characterises \emph{tame moduli spaces}, but in our situation the two notions coincide since $\chara \kk =0$.
Note that since we assume our stacks and varieties to be of finite type over $\kk$,
maps between them are automatically quasi-compact.

\begin{proposition}[{\cite[Thm.~3.2]{Abramovich-Olsson-Vistoli},\cite{Keel-Mori},\cite[Rem.~7.3]{Alper}}]
\label{prop:properties_pi_coarse}
Let $\pi\colon\XX \to X$ be the natural map to the coarse moduli space. Then 
\begin{enumerate}
\setcounter{enumi}{2}
\item \label{it:exact} $\pi_*\colon\QCoh  \XX \to \QCoh X$ is exact;
\item \label{it:struct} there is a natural isomorphism $\OO_X \to \pi_* \OO_\XX$.
\end{enumerate}
Conversely, if $\pi\colon\XX\to X$ satisfies the conditions \eqref{it:iso_pts}, \eqref{it:exact} and \eqref{it:struct}, 
then $\pi$ is a coarse moduli map.
\end{proposition}

Now we recall the concept of the rigidification of a stack. We refer to \cite{Abramovich-Corti-Vistoli}, \cite{Abramovich-Olsson-Vistoli} and \cite{Fantechi-etal} for details.

Let $\XX$ be a stack. The \emph{inertia stack} of $\XX$ is defined to be the fibre product $I(\XX) \coloneqq \XX \times_{\XX \times \XX} \XX$ . 
The inertia stack of a smooth Deligne-Mumford stack is smooth, but will consist of several components, not all necessarily of the same dimension. Let $I(\XX) \to \XX$ be the natural projection. Its identity section gives an irreducible component canonically isomorphic to $\XX$; all other components are called \emph{twisted sectors}. 

A smooth Deligne-Mumford stack of dimension $d$ is an \emph{orbifold} if and only if all the twisted sectors have dimension at most $d-1$, and is \emph{canonical} if and only if all twisted sectors have dimension at most $d-2$. 

The canonical stack can be defined also as the smooth Deligne-Mumford stack $\Xcan$, whose natural map to the coarse moduli space $\pi\colon \Xcan \to X$ fails to be an isomorphism just on a locus of dimension $\leq d-2$. 
Note that $\Xcan$ is unique up to isomorphism with this property; see~\cite[Thm.~4.6]{Fantechi-etal}.

The \emph{generic stabiliser} $I^{gen}(\XX)$ of $\XX$ is the subsheaf of groups of $I(\XX)$ given by the union of all $d$-dimensional components of $I(\XX)$. 

Let now $H \subset I(\XX)$ be a flat subgroup stack. 
Let $\xi \in \XX(T)$ be an object of the stack $\XX$ over a $\kk$-scheme $T$, which corresponds to a morphism $T \to \XX$. 
The pullback $T \times_{\XX} I(\XX)$ is canonically isomorphic to 
the group scheme $\mathcal{A}ut_{T,\XX}(\xi)$ and the pullback of $H$ to $T$ gives a flat subgroup scheme 
$H_\xi \subset \mathcal{A}ut_{T,\XX}(\xi)$.
It is easy to see that the construction implies that $H_\xi$ is normal in $\mathcal{A}ut_{T,\XX}(\xi)$.
In this situation the following theorem defines the \emph{rigidification} of $\XX$ by $H$. 

\begin{proposition}[{\cite[Thm~A.1.]{Abramovich-Olsson-Vistoli}}]  \label{prop.rigidification}
There is a stack $\XX \thickslash H$ called the \emph{rigidification} of $\XX$ by $H$ and 
a morphism $\nu\colon \XX \to  \XX \thickslash H$ satisfying the following properties:
\begin{enumerate}
\item for any $\kk$-scheme $T$ and for any object $\xi \in \XX(T)$, the homomorphism of group schemes 
$\mathcal{A}ut_{T,\XX}(\xi) \to \mathcal{A}ut_{T,\XX \thickslash H}(\nu(\xi))$ is surjective with kernel $H_\xi$, 
\item $\XX$ is a fppf gerbe over $\XX \thickslash H$
(for the definition of gerbe, see Section~\ref{sec:gerbes}).
\end{enumerate}
Furthermore, if $\XX$ is a Deligne-Mumford stack, then also $\XX \thickslash H$ is a Deligne-Mumford stack.
\end{proposition}

It can be shown that these properties characterise $\XX \thickslash H$ uniquely up to isomorphism. 

A more intuitive way to think of the rigidification is to consider first the prestack
$(\XX \thickslash H)^{pre}$, whose objects are the same as those of $\XX$ and 
whose sheaf of automorphisms $\mathcal{A}ut_{T, (\XX \thickslash H)^{pre}}(\xi)$ of an object $\xi \in \XX(T)$ 
is the quotient of $\mathcal{A}ut_{T, \XX}(\xi)$ by the subsheaf of normal 
groups $H_\xi$. The rigidification $\XX \thickslash H$ can be obtained as the fppf stackification of this prestack.
From this construction follows that $\XX$ and $\XX \thickslash H$ have the same coarse moduli space.

In our work we are mainly interested in the rigidification of a stack by its generic stabiliser, 
which we assume to be abelian and flat.
It will be called \emph{the} rigidification of $\XX$ and indicated by $\Xrig$. 

\begin{proposition}[{\cite[Rem.~3.7]{Fantechi-etal}}]
Let $\XX$ be a Deligne-Mumford stack and $\pi\colon\XX\to X$ its coarse moduli map.
Moreover, let $D$ be a prime divisor on $X$.
Denote by $\DD$ the reduced closed substack with support $\pi^{-1}(D)$.
When restricting to the smooth locus of $X$, $D$ becomes Cartier, and there is a unique positive integer $a$ such that $\pi^{-1}(D \cap X_{\sm}) = 
a (\DD \cap \pi^{-1}(X_{\sm}))$.
\end{proposition}

\begin{definition}
In the notation above we call $a$ the \emph{divisor multiplicity} of $\DD$.
The \emph{ramification divisor} of $\XX$ is the sum of all prime divisors of $\XX$ with positive divisor multiplicity. 

Now let $\XX$ be an arbitrary Deligne-Mumford stack and $\nu\colon\XX\to\Xrig$ its rigidification. If $\DD^\rig$ is the ramification divisor of $\Xrig$, we call the reduced closed substack $\DD$ inside $\nu^{-1}(\DD^\rig)$ the ramification divisor of $\XX$.
\end{definition}

\begin{definition}
A line bundle $\LL$ on a stack $\XX$ is called \emph{ample} if
$H^i(\XX,\FF \otimes \LL^{\otimes n})$ vanishes for any sheaf $\FF$ on $\XX$, $i>0$ and $n \gg 0$.
\end{definition}

\begin{lemma}
\label{lem:pullback_ample_lb}
Let $\pi\colon \XX \to X$ be the map of a Deligne-Mumford stack to its coarse moduli space and $L$ an ample line bundle on $X$.
Then $\pi^* L$ is an ample line bundle on $\XX$.
\end{lemma}

\begin{proof}
Since $L$ is locally free, we can use the projection formula to get 
$\pi_* ( F \otimes \pi^* L^{\otimes n} ) = \pi_* F \otimes L^{\otimes n}$, 
since $\pi_*$ is exact by Proposition~\ref{prop:properties_pi_coarse}\eqref{it:exact}.
Apply on both sides $H^i(X,\blank)$ this becomes 
$H^i (\XX, F\otimes \pi^* L^{\otimes n} ) = H^i (X, \pi_* F \otimes L^{\otimes n} )$. 
But the latter cohomology group vanishes for $i>0$ and $n\gg0$.
\end{proof}

\subsection{Root constructions}
\label{sec:root_constr}

In this section, we recall the definition of two root constructions as found in \cite{Cadman} and \cite{Abramovich-Graber-Vistoli} and we state some facts about these.

\begin{definition}
Let $\XX$ be a stack and $\DD$ an effective divisor on $\XX$.
Let $r$ be a positive integer.
Consider the fibre product
\[
\xymatrix@R=3ex{
\XX \times_{[\AA^1/\kk^*]} [\AA^1/\kk^*] \ar[r] \ar[d] & [\AA^1/\kk^*] \ar[d]^{r}\\
\XX \ar[r]_{\DD} & [\AA^1/\kk^*]
}
\]
where the lower map corresponds to the divisor and the  right map is induced by $p \mapsto p^r$.
The stack $\sqrt[r]{\DD/\XX} \coloneqq \XX \times_{[\AA^1/\kk^*]} [\AA^1/\kk^*]$
is the \emph{$r$-th root of the divisor $\DD$ on $\XX$}.
\end{definition}

\begin{definition}
Let $\XX$ be a stack and $\LL$ a line bundle on $\XX$.
Let $r$ be a positive integer.
Consider the fibre product
\[
\xymatrix@R=3ex{
\XX \times_{\BB \kk^*} \BB \kk^* \ar[r] \ar[d] & \BB \kk^* \ar[d]^{r}\\
\XX \ar[r]_{\LL} & \BB \kk^*
}
\]
where the lower map corresponds to the line bundle and the right map is induced by $p \mapsto p^r$.
The stack $\sqrt[r]{\LL/\XX} \coloneqq \XX \times_{\BB \kk^*} \BB \kk^*$ is the \emph{$r$-th root of the line bundle $\LL$ on $\XX$}.
\end{definition}

\begin{definition}
Let $\XX$ be a stack. We say that $\XX'$ is \emph{obtained by roots} from $\XX$, if $\XX'$ is the result of performing a finite sequence of root constructions 
with divisors or line bundles starting from $\XX$.
\end{definition}

By \cite[Ex.~7.21]{Vistoli}, a coherent sheaf on a quotient stack $\XX= [Z/G]$ is a $G$-equivariant sheaf on $Z$, i.e., a coherent sheaf $L_Z$ on $Z$ equipped, for every 
$g \in G$, with an isomorphism $\varphi_g\colon L_Z \to g^*L_Z$, such that $\varphi_{gh} = h^* \varphi_g \circ \varphi_h$. The data $\{\varphi_g\}_{g \in G}$ is called a \emph{$G$-linearisation} of $L_Z$.

In the case of invertible sheaves this leads to the following exact sequence; see \cite[Lem.~2.2]{Knop-Kraft-Vust}:
\[ 0 \to H^1_{alg}(G, \OO^*_Z) \to \Pic(\XX) \to \Pic(Z). \]
Here, the group $H^1_{alg}(G, \OO^*_Z)$ parametrises $G$-linearisations of the trivial line bundle on $Z$.
The group of $G$-linearisations is naturally linked to the group of characters of $G$ by the following exact sequence 
obtained from \cite[Prop.~2.3]{Knop-Kraft-Vust} 
\[1 \to H^0(\XX, \OO^*_{\XX}) \to H^0(Z, \OO^*_Z) \to \Hom(G, \kk^*) \to \Pic(\XX) \to \Pic(Z).   \]
This implies that, if $H^0(Z, \OO^*_Z)= \kk^*$, the first arrow is an isomorphism and the group $H_{alg}^1(G, \OO^*_Z)$ coincides with the group of characters of $G$.

In this case we can generalise {\cite[Lem.~7.1]{Fantechi-etal}}:

\begin{proposition}
\label{prop:root_description_quot}
Let $Z$ be a variety with $H^0(Z, \OO^*)=\kk^*$ and $G$ an abelian group acting on it, such that $\XX = [Z/G]$ is a Deligne-Mumford stack. 
Let $\DD$ be a divisor on $\XX$ corresponding to a line bundle $\LL$ and a section $s$. 
Assume that $\LL$ is given by the trivial line bundle on $Z$ and a character $\chi$ of $G$.  
Finally, let $r$ be a positive integer.
\begin{enumerate}
\item The stack $\sqrt[r]{\DD/\XX}$ is isomorphic to $[Z'/G']$ where both $Z'$ and $G'$ are defined as the fibre products
\[
\xymatrix@R=3ex{
Z' \ar[r] \ar[d] & \AA^1 \ar[d]^r\\
Z  \ar[r]_{s} & \AA^1
}
\qquad
\xymatrix@R=3ex{
G' \ar[r] \ar[d] & \kk^* \ar[d]^r\\
G  \ar[r]_{\chi} & \kk^*
}
\]
and both maps on the right are induced by $p \mapsto p^r$.
The action of $G'$ on $Z'$ is given by
\[
(g,\lambda) \cdot (z,t) = (g\cdot z, \lambda t).
\]
\item The stack $\sqrt[r]{\LL/\XX}$ is isomorphic to $[Z/G']$ where $G'$ is defined as above and acts on $Z$ using the map $G' \to G$.
\end{enumerate}
\end{proposition}

\begin{remark}
The variety $Z'$ together with the map $Z'\to Z$ is called the $r$-th \emph{cyclic cover} of $Z$ ramified along $D$, which is the divisor $\DD$ pulled back to $Z$.

Note that, if $\Pic(Z)=0$, any invertible sheaf on the quotient stack $\XX$ is given by a unique character of $G$.
\end{remark}


\begin{proposition}[{\cite[Thm.~2.3.3 and Rem.]{Cadman}}]
Let $\XX$ be a stack obtained by roots from a Deligne-Mumford stack.
Then $\XX$ is also a Deligne-Mumford stack.
\end{proposition}

\begin{proposition}[{\cite[Prop.~3.10(vii)]{Alper}}]
\label{prop:cohom_affine_basechange}
Consider the following cartesian square of Artin stacks:
\[
\xymatrix@R=3ex{
\XX' \ar[r] \ar[d]_{f'} & \YY' \ar[d]^{f}\\
\XX \ar[r] & \YY
}
\]
If $f$ is cohomologically affine (i.e., quasi-compact and $f_*$ is exact) and $\YY$ has quasi-affine diagonal over $\kk$, then $f'$ is also cohomologically affine.
Note that the condition on $\YY$ is automatically fulfilled if $\YY$ is a Deligne-Mumford stack.
\end{proposition}

\begin{lemma}
\label{lem:root_keeps_coarse}
Let $\XX \to X$ be the map to the coarse moduli space of a Deligne-Mumford stack $\XX$,
and $\XX'$ obtained by roots with divisors from $\XX$.
Then the composition $\XX' \to \XX \to X$ is the coarse moduli map of $\XX'$.
\end{lemma}

\begin{proof}
Let $\XX'$ be obtained by one single root  of a divisor from $\XX$. So $\XX'$ fits into the following cartesian square:
\[
\xymatrix@R=3ex{
\XX' \ar[r] \ar[d]_{\nu} & [\AA^1/\kk^*] \ar[d]^{r}\\
\XX \ar[r]             & [\AA^1/\kk^*]
}
\]
We will check that $\XX' \xrightarrow{\nu} \XX \to X$ is a coarse moduli map, using the characterisation in Proposition~\ref{prop:properties_pi_coarse}. Since this already holds for $\XX \to X$, we are left to show that $\nu$ is cohomologically affine, induces an isomorphism $\OO_{\XX} \to \nu_* \OO_{\XX'}$ and a bijection $|\XX'(\kk)| \to |\XX(\kk)|$.

First we show that $\nu$ is cohomologically affine by applying Proposition~\ref{prop:cohom_affine_basechange}. For this note that $[\AA^1/\kk^*]$ has a quasi-affine diagonal over $\kk$. Moreover, the composition $[\AA^1/\kk^*] \xrightarrow{r} [\AA^1/\kk^*] \to \Spec \kk$ is the same map as $[\AA^1/\kk^*] \to \BB \kk^* \to \Spec \kk$, which is cohomologically affine.
Therefore, by \cite[Prop.~3.14]{Alper}, $[\AA^1/\kk^*] \xrightarrow{r} [\AA^1/\kk^*]$ is also cohomologically affine and consequently $\nu$ as well.

Secondly, $\OO_{\XX} \to \nu_* \OO_{\XX'}$ is an isomorphism by \cite[Thm.~3.1.1(3)]{Cadman}. 

Finally, we need to see whether $\nu$ induces a bijection of closed points
$|\XX'(\kk)| \to |\XX(\kk)|$.
Let $\pt \to \XX$ be a closed point and consider the fibre product
\[
\xymatrix@R=3ex{
P \ar[r] \ar[d]  & \XX' \ar[d] \ar[d] \\
\pt \ar[r] & \XX 
}
\]
By extending this cartesian square by the diagram above, one sees that $P$ is either $\pt$ or $[\Spec\left({}^{\kk[t]}/_{t^r}\right)/\mu_r]$, which has $\pt$ or $\Spec\left({}^{\kk[t]}/_{t^r}\right)$ as an atlas. Note that latter containes a unique closed point.
Therefore, $|\XX'(\kk)| \to |\XX(\kk)|$ is surjective.

Suppose we have two maps $f_i\colon \pt \to \XX'$ with $i=1,2$ such that the composition $\pt \to \XX' \to \XX$ is the same map $g$ for both up to isomorphism.
Therefore, we can build the following diagram where the square is cartesian
\[
\xymatrix@R=3ex{
\pt \ar@/^/[rrd]^{f_i}\ar@/_/[ddr]_{\id} \ar@{-->}[rd] \\
& P \ar[r] \ar[d]  & \XX' \ar[d] \ar[d] \\
& \pt \ar[r]_{g} & \XX
}
\]
Hence, the $f_i$ factor over $P$. Independent of its actual form, $|P(\kk)| = \pt$, so the maps $f_i$ differ only by an isomorphism. This shows that $|\XX'(\kk)| \to |\XX(\kk)|$ is also injective.
\end{proof}

\begin{remark}
A very similar proof also works for roots with line bundles. We will show the analoguous statement in a more general situation, namely for gerbes, 
see Example~\ref{ex:lb_roots_coarse}.
\end{remark}

\begin{lemma}[{\cite[Prop.~3.1]{Pardini}}]
\label{lemma.snc-smooth}
Let $Z$ be a smooth variety and $D_1,\ldots,D_l$ prime divisors on $Z$.
Moreover, let $Z'$ be obtained from $Z$ by cyclic covers along the divisors $D_1,\ldots,D_l$.
Then $Z'$ is smooth if and only if $D_1,\ldots,D_l$ are simple normal crossing.
\end{lemma}

\begin{corollary}
\label{cor:smooth_iff_snc}
Let $\XX$ be a smooth Deligne-Mumford stack and $\XX'$ obtained by roots along the divisors $D_1,\ldots,D_l$.
Then $\XX'$ is smooth if and only if $D_1,\ldots,D_l$ are normal crossing.
\end{corollary}
\begin{proof}
Note that the smooth Deligne-Mumford stack $\XX$ can be \'etale locally covered by smooth varieties. Moreover, by Proposition \ref{prop:root_description_quot}, the root on such a variety $U$ along some of its divisors yields a quotient stack, whose atlas is a cyclic cover of $U$. The application of Lemma \ref{lemma.snc-smooth} concludes the proof. 
\end{proof}

\subsection{Gerbes}
\label{sec:gerbes}

In this section we recall the definition of gerbes and some of their properties. 
We mainly follow \cite{Breen} and we refer to \cite[Ch.~IV.2]{Giraud} for a complete treatment.

\begin{definition}
A \emph{gerbe} on a stack $\XX$ is a stack in groupoids $\mathcal G$ on $\XX$ which is:
\begin{enumerate}
\item \emph{locally non-empty}: for any \'etale morphism $U \to \XX$ of $\XX$, ${\mathcal G} (U)$ is non-empty; 
\item \emph{locally transitive}: for any $x, y \in {\mathcal G} (U)$ there is an open cover $\{V_i \to U\}$, 
such that $x|_{V_i}, y|_{V_i} \in {\mathcal G} (V_i)$ are isomorphic. 
\end{enumerate}
Let $G$ be a sheaf of abelian groups over $\XX$. 
A gerbe $\mathcal G$ on $X$ is called \emph{$G$-banded} (or simply a \emph{$G$-gerbe}) if, 
for every \'etale atlas $U \to \XX$ and for every object $x \in \mathcal{G}(U)$, 
there exists an isomorphism $\alpha_x\colon G|_U \to {\mathcal A}ut_{U, {\mathcal G}}(x)$ of sheaves of groups and all these isomorphisms are compatible with the pullbacks in the following sense. 
Given two objects $x \in \mathcal{G}(U)$ and $y \in \mathcal{G}(V)$ and an arrow $\phi\colon x \to y$ over a morphism of schemes $f\colon U \to V$, the natural pullbacks:
\[ \phi^*\colon  {\mathcal A}ut_{V, {\mathcal G}}(y) \to  {\mathcal A}ut_{U, {\mathcal G}}(x) \quad \mbox{and} \quad f^*\colon\mathcal{G}(V) \to \mathcal{G}(U)  \]
commute with the isomorphisms, i.e., $\alpha_x \circ f^* = \phi^* \circ \alpha_y$.
\end{definition}

\begin{definition}
A gerbe ${\mathcal G}$ on $\XX$ is \emph{trivial} if the fibre category ${\mathcal G} (\XX)$ is non-empty. 
\end{definition}

\begin{example}
The stack $\BB_\XX G$ of $G$-principal bundles on $\XX$ is a $G$-gerbe and it is 
\emph{the} trivial gerbe, i.e., every trivial $G$-gerbe on $\XX$ is isomorphic to it. 
\end{example}

\begin{remark}
By definition each gerbe is locally non-empty, and therefore locally trivial.   
\end{remark}

\begin{example}
\label{ex:lb_roots_coarse}
Let $\XX$ be a stack and $\XX\thickslash H$ an $H$-rigidification of it, 
where $H$ is a flat subgroup stack of the inertia of $\XX$. Then $\XX \to \XX \thickslash H$ is an $H$-gerbe.  
Let $\XX'$ be the $r$-th root of a line bundle over $\XX$. The stack $\XX'$ is a gerbe banded by the constant 
sheaf $\boldsymbol{\mu}_r$, especially $\XX = \XX' \thickslash \boldsymbol{\mu}_r$. Therefore $\XX'$ and $\XX$ have the same coarse moduli space.
\end{example}

An important result about $G$-gerbes is their classification up to isomorphisms that preserve the $G$-banded structure. 
By \cite[Sec.~IV.3.4]{Giraud}, the group $\Het^2(\XX,G)$ classifies the equivalence classes of $G$-gerbes on $\XX$. 

Let $\Gm$ be the constant sheaf to $\kk^*$ and let
\[ 0 \to \boldsymbol{\mu}_r \xrightarrow{\iota} \Gm \xrightarrow{\wedge r} \Gm \to 0  \]
be the Kummer sequence. It induces a long exact sequence 
\[ \ldots \to \Het^1(\XX, \Gm) \to \Het^2(\XX, \boldsymbol{\mu}_r) \xrightarrow{\iota^*} \Het^2(\XX, \Gm)  \to \ldots  .  \]
The $\boldsymbol{\mu}_r$-gerbes whose $\Het^2(\XX, \boldsymbol{\mu}_r)$-class has zero image via $\iota^*$ in $\Het^2(\XX, \Gm)$ are called \emph{essentially trivial} and, from the exact sequence, they are exactly the ones obtained as $r$-th roots of a line bundle on $\XX$.     

Let $\mathcal{G}$ be a $\boldsymbol{\mu}_r$-gerbe on $\XX$. The Leray spectral sequence for the \'etale sheaf $\Gm$ with respect to the map $\nu\colon \mathcal{G} \to \XX$ gives the exact sequence:
\begin{equation} 
\label{Leray for gerbes} 
\tag{\ensuremath{\spadesuit}}
0 \to \Het^1(\XX, \Gm) \xrightarrow{\nu^*} \Het^1(\mathcal{G}, \Gm) \xrightarrow{\mathrm{res}} \Het^0(\XX,R^1\nu_* \Gm) \xrightarrow{\mathrm{obs}} \Het^2(\XX, \Gm) \to \ldots  
\end{equation} 
Here $\Het^1(\mathcal{G}, \Gm) = \Pic(\mathcal{G})$, $\Het^0(\XX,R^1\nu_* \Gm)=\Pic(\BB \mu_r)=\ZZ/r\ZZ$ 
since the the fibre of $\nu$ is isomorphic to $\BB \mu_r$,
and the map  
$\mathrm{res}$ is the restriction of line bundles to this fiber. The map $\mathrm{obs}\colon\ZZ/r\ZZ  \to \Het^2(\XX, \Gm)$ sends $\bar 1$ to the image by $\iota^*$ of the class of the $\boldsymbol{\mu}_r$-gerbe 
$\mathcal{G}$. Thus it is zero if and only if the gerbe is essentially trivial and in this case the sequence becomes short exact.

\section{Construction of Mori dream stacks}

In the following, we propose the definition of a Mori dream stack, similar to the one of \MD-spaces. Before we give the definition, we make two preliminary observations.
First, since \MD-spaces have at most quotient singularities, we can associate to any \MD-space a smooth Deligne-Mumford stack which will be smooth.
Moreover for a stack $\XX$, we change the definition of its Cox ring to $\RR(\XX) = \bigoplus_{\LL\in\Pic(\XX)} H^0(\XX,\LL)$ (see Remark~\ref{rem:cl-pic} for a reason).
The structure as an algebra is obtained in a similar way as in Section~\ref{sec:md_spaces}:

If $\Pic(\XX)$ is finitely generated as a $\ZZ$-module, we can fix an isomorphism $\Pic(\XX) \cong \ZZ/r_1\ZZ \times\cdots\times \ZZ/r_k\ZZ$ with $r_i \geq 0$. To avoid confusion, we note that $r_i$ can be zero, in which case $\ZZ/r_i\ZZ = \ZZ$.
For each $1\leq i \leq k$, choose a line bundle $\LL_i$ that generates the corresponding $\ZZ/r_i\ZZ$, and moreover an isomorphism $\sigma_i\colon \LL^{r_i} \xrightarrow{\sim} \OO_\XX$.

Note that any line bundle $\LL \in \Pic(\XX)$ can be written uniquely up to isomorphism as $\bigotimes \LL_i^{a_i}$ with $0 \leq a_i$ and $a_i < r_i$ for $r_i\neq0$. Moreover, for $r_i\neq0$ the chosen isomorphisms $\sigma_i$ induce isomorphisms $\LL_i^{a_i} \otimes \LL_i^{b_i} \xrightarrow{\sim} \LL_i^{c_i}$, where $0 \leq a_i,b_i, c_i < r_i$ and $ a_i + b_i = c_i \in \ZZ/r_i\ZZ$. 
These data define the sheaf of algebras
$\overline\RR= \bigoplus_{\LL\in \Pic(\XX)} \LL$ on $\XX$.

\begin{definition}
Let $\XX$ be a stack such that $\Pic(\XX)$ is finitely generated.
The \emph{Cox ring} of $\XX$ is the $\Pic(\XX)$-graded ring
\[\RR(\XX) = H^0(\XX, \overline\RR).\]
\end{definition}

\begin{remark}
By \cite[Prop.~4.7.3]{SGA1} and \cite[Prop.~4.1.]{SGA8}, the sheaf of $\OO_\XX$-algebras $\overline\RR$ defines a $\Hom(\Pic(\XX),\kk^*)$-torsor over $\XX$.

Actually, there are many ways to define the algebra structure on $\RR(\XX)$ by choosing different (not necessarily minimal) generating sets of $\Pic(\XX)$. All of them yield isomorphic Cox rings.

Moreover, we want to note here that we can also define such a $\Pic$-graded ring for a variety $X$. If $X$ is $\QQ$-factorial, this will be then a subring of the $\Cl$-graded Cox ring $\RR(X)$.
\end{remark}

\begin{definition}
We call a Deligne-Mumford stack $\XX$ a \emph{Mori dream stack} if 
\begin{enumerate}
\item $\XX$ is smooth, 
\item $H^0(\XX,\OO_\XX^*) = \kk^*$, 
\item $\Pic(\XX)$ is finitely generated as a $\ZZ$-module,
\item $\RR(\XX)$ is finitely generated as a $\kk$-algebra,
\end{enumerate}
\end{definition}

In the following we will always use the abbreviation \emph{\MD-stack}.

\begin{definition}
We call a Mori dream stack $\XX$ a \emph{Mori dream quotient stack} (or \emph{\MD-quotient stack} for short) if
\[
\XX = \stackquot{\Spec \RR(\XX) \setminus V(J_{\irr}) }{ \Hom(\Pic(\XX),\kk^*) } ,
\]
where $J_{\irr}$ is defined as for \MD-spaces.
\end{definition}

\begin{remark}
\label{rem:cl-pic}
For an orbifold, the Picard group and the divisor class group coincide. In the general case, the divisor class group can be strictly contained in the Picard group.
As an example, for a positive integer $r$, the stack $\BB \boldsymbol{\mu}_r$ is an \MD-quotient stack.
Here $\Pic(\BB \boldsymbol{\mu}_r) = \ZZ/r\ZZ$ which gives the quotient description, but $\Cl(\BB \boldsymbol{\mu}_r) = 0$.
\end{remark}

\begin{remark}
\label{rem:Xcan_smooth}
For a reductive group scheme $G$ acting linearly on some scheme $Z$, we can look at the map $[Z^{ss}/G] \to Z^{ss}/G$.
If $[Z^{ss}/G]$ is a Deligne-Mumford stack, then the scheme $Z^{ss}/G$ is already the coarse moduli space, see for example \cite{Alper}.

Therefore an \MD-space $X = Z^{ss} /G$ with $Z = \Spec \RR(X)$ and $G = \Hom(\Cl(X),\kk^*)$ gives a smooth Deligne-Mumford stack $\XX = [Z^{ss}/G]$ if $Z^{ss}$ is smooth. Moreover in this case, $\XX$ is the canonical stack of $X$, as the singular locus of $X$ has codimension at least $2$, and $\XX \to X$ is an isomorphism outside the singular locus. 

Moreover it is a classical fact that every variety
with finite quotient singularities is the coarse moduli space of a canonical smooth Deligne-Mumford stack unique up to isomorphism; see~\cite[Rem.~4.9]{Fantechi-etal}. 
\end{remark}

Our aim for the rest of the section is to study how the definitions of \MD-stacks and \MD-quotients behave with respect to the constructions we recalled in the previous section. First we describe the canonical stack of an \MD-space and then we prove that root constructions and gerbes preserve the 
property of being an \MD-stack.
\begin{theorem} 
\label{theorem.canonical MD-stack}
Let $X$ be an \MD-space. Then its canonical stack $\Xcan$ is an \MD-stack and $\RR(X) =  \RR(\Xcan)$. 
Moreover, suppose that $\Spec \RR(X) \setminus V(J_{\irr})$ is smooth. Then its canonical stack $\Xcan$ is an \MD-quotient, so
\[
\Xcan = \stackquot{\Spec \RR(\Xcan) \setminus V(J_{\irr}) }{  \Hom(\Pic(\Xcan),\kk^*)}.
\]
\end{theorem}
\begin{proof}
First, the canonical stack $\Xcan$, which is smooth by definition, exists by Remark~\ref{rem:Xcan_smooth}.
By construction, we get $\Pic(\Xcan) = \Cl(\Xcan) = \Cl(X)$ which is finitely generated. 
Moreover, we note  that $H^0(X,\OO(D)) = H^0(\Xcan,\OO(D))$ for any divisor class $D$, by \cite[Sec.~1.2]{Fantechi-etal}.

Therefore, $H^0(\Xcan,\OO_{\Xcan}^*) = \kk^*$ and $\RR(\Xcan) = \RR(X)$ as algebras. For the latter note that, if we fix generators $\LL_i$ of $\Pic(\Xcan)$ as in the construction of the Cox ring $\RR(\Xcan)$, then we can write an analoguous resolution $K \onto \Cl(X)$ by using the isomorphism $\Pic(\Xcan)=\Cl(X)$, and vice versa.
Hence $\Xcan$ is indeed an \MD-stack.

By Lemma~\ref{lem:pullback_ample_lb} it is immediate that $J_{\irr}(\Xcan) = J_{\irr}(X)$. If we assume additionally that $\Spec \RR(X) \setminus V(J_{\irr}) = \Spec \RR(\Xcan) \setminus V(J_{\irr})$ is smooth,  
then $\Xcan$ is even an \MD-quotient by Remark~\ref{rem:Xcan_smooth}.

\end{proof}

In the following example we see that the assumption on the smoothness of  $\Spec \RR(X) \setminus V(J_{\irr}(X))$ is indeed necessary to have the description of the canonical stack as an \MD-quotient.

\begin{example}
\label{ex:dicyclic}
Let $\Dic_n$ be the \emph{dicyclic} or \emph{binary dihedral} group of $4n$ elements, defined as: 
\[
\Dic_n = \Genby{ \begin{pmatrix} 0& 1\\-1& 0 \end{pmatrix}, \begin{pmatrix} \xi & 0\\0& \xi^{-1} \end{pmatrix} } = \genby{ s,r \mid s^4 = 1 = r^{2n}, r^n = s^2, srs^{-1} = r^{-2} }
\]
for a primitive $2n$-th root $\xi$ of unity and  
let $X=\AA^2/\Dic_n$ be the quotient by the usual matrix action. 
It turns out that $X$ is an \MD-space. First, it has an isolated quotient singularity at zero. Easy calculations show that 
\[
\Cl(X) = 
\begin{cases}
\ZZ/2\ZZ \oplus \ZZ/2\ZZ & \text{if $n$ even,}\\
\ZZ/2\ZZ & \text{if $n$ is odd,}
\end{cases}
\] 
$\RR(X)= \kk[u,v,w]/(u^n - vw)$ and that invertible global sections are just the non-zero constants. 
As $V(J_{\irr})$ is empty in this case, we computed that $X= \Spec \RR(X) / \Hom(\Cl(X), \kk^*) = \Spec \RR(X)^{\Cl(X)}$.

On the other hand, since $\RR(X)$ is singular at zero, the quotient stack $\XX= [\Spec \RR(X)/ \Hom(\Cl(X), \kk^*)]$ is not smooth. 
Hence $\XX$ differs from the (smooth) canonical stack $\Xcan=[\AA^2/\Dic_n]$. 

The computation of the proof above shows that $\Xcan$ is an \MD-stack with $\Pic(\Xcan) =\Cl(X)$ and $\RR(\Xcan) =\RR(X)$. Consequently, there is no description of $\Xcan$ as an \MD-quotient, since the candidate would be the singular $\XX$.
\end{example}

\begin{theorem} \label{theorem.divisor root preserves MD}
Let $\XX$ be an \MD-stack and $\XX'$ be a stack obtained by root constructions along simple normal crossing divisors from $\XX$. 
Then $\XX'$ is an \MD-stack and its Cox ring is of the form 
\[
\RR(\XX') = \RR(\XX)[z_1,\ldots,z_l]/(z_i^{r_i}-s_i \mid i=1,\ldots,l)
\]
for some positive integers $r_i$ and $s_i \in \RR(\XX)$.
\end{theorem}
\begin{proof}
Smoothness is ensured by Corollary~\ref{cor:smooth_iff_snc}. 

Without loss of generality, we restrict to the case of a single smooth divisor $\DD$ on $\XX$ and show the remaining properties  of \MD-stacks for the stack $\nu\colon\XX' = \sqrt[r]{\DD/\XX} \to \XX$.

By \cite[Thm.~3.1.1]{Cadman} we have the following pushout-diagram for the Picard groups:
\[
\xymatrix@R=3ex@C=1ex{
1 \ar@{|->}[d] & \ZZ \ar[d] \ar[rr]^{\cdot r} && \ZZ \ar[d]\\
\OO(\DD) & \Pic(\XX) \ar[rr]_{\nu^*} && \Pic(\XX') \ar@{=}[r]^-{\sim} &  {}^{\Pic(\XX) \oplus \ZZ}/_{\ZZ\cdot(\OO(\DD),-r)}
}
\]
The line bundle $\TT$ on $\XX'$ corresponding to $(\OO,1)$ is called the tautological line bundle. Moreover, there is an isomorphism $\sigma\colon \TT^r \xrightarrow{\sim} \nu^* \OO(\DD)$ and a tautological section $\tau$ of $\TT$ such that $\sigma(\tau^r) = \nu^*(s)$, where $s$ is the section cutting out $\DD$.

Consequently, for any line bundle $\LL'$ on $\XX'$, there exists a line bundle $\LL$ on $\XX$, unique up to isomorphism, and a unique integer $0\leq k< r$, such that
$\LL' \cong \nu^*\LL \otimes \TT^k.$
Moreover by \cite[Cor.~3.1.3]{Cadman}, the multiplication by $\tau^k$ gives an isomorphism between global sections: 
\[
H^0(\XX,\LL) \xrightarrow{\sim} H^0(\XX',\nu^*\LL) \xrightarrow[\cdot \tau^k]{\sim} H^0(\XX',\LL'). 
\]
By \cite[Thm.~3.1.1(3)]{Cadman}, there is an isomorphism  $\nu_*\OO_{\XX'} \cong \OO_{\XX}$, therefore $H^0(\XX', \OO^*_{\XX'}) = H^0(\XX, \OO^*_{\XX})=\kk^*$.

Next we show that the Cox ring $\RR(\XX')$ is finitely generated.
Choose line bundles $\LL_1,\ldots,\LL_k$ on $\XX$ and isomorphisms $\sigma_{i}\colon \LL_i^{s_i} \xrightarrow{\sim} \OO_{\XX}$ to define the algebra structure on the Cox ring of $\XX$. 
Since $\Pic(\XX')$ is generated by $\nu^* \LL$ and $\TT$,
we can use $\nu^* \sigma_i$ and $\sigma\colon\TT^r \xrightarrow{\sim} \nu^* \OO(\DD)$ to obtain the algebra structure on $\RR(\XX')$.
With these data the Cox Ring of $\XX'$ is $\RR(\XX') = \RR(\XX)[z]/(z^r-s)$.

Hence $\XX'$ is indeed an \MD-stack. 
\end{proof}

\begin{theorem} \label{thm.gerbes preserve MD}
Let $\XX$ be a stack and $\nu\colon\XX'\to \XX$ an $H$-gerbe on $\XX$, where $H$ is an abelian group scheme. Then $\XX$ is an \MD-stack if and only if $\XX'$ is an \MD-stack. Moreover in this case, their Cox rings are isomorphic.
\end{theorem}
 
\begin{proof}
The local triviality of gerbes ensures that $\XX$ is smooth if and only if $\XX'$ is smooth. 
 
As in \eqref{Leray for gerbes}, the Leray spectral sequence for $\nu\colon \XX' \to \XX$ yields the long exact sequence:
\[  0 \to \Pic(\XX) \xrightarrow{\nu^*} \Pic(\XX') \to \Pic(\BB H) \to H^2(\XX, \mathbb{G}_m) \]
and $\Pic(\XX')$ turns out to be an extension of $\Pic(\XX)$ with a subgroup of the finitely generated group $\Pic(\BB H)$. 
Therefore, if one of the Picard groups of $\XX$ and $\XX'$ is finitely generated as a $\ZZ$-module, the other is, too.  

Locally, let $\XX =\Spec A$. Since $\XX'$ is an $H$-gerbe on $\XX$, it is locally trivial, 
and, without loss of generality, we can assume that $\XX' \times_{\XX} \Spec A \cong \BB H \times \Spec A$. Locally, the exact sequence above becomes
\[  0 \to \Pic(\XX) \xrightarrow{\nu^*} \Pic(\XX') \to \Pic(\BB H) \to 0. \]
The line bundles on $\XX'$ of the form $\LL' = \nu^* \LL$, with $\LL \in \Pic(\XX)$, correspond to graded $A$-modules concentrated in degree zero, and for them we have $H^0(\XX, \LL) =H^0(\XX', \LL')$. 
Let now $\LL'$ be a line bundle on $\XX'$ which is not mapped to zero in $\Pic(\BB H)$. Thus it corresponds to a graded $A$-module with trivial degree zero part and therefore $H^0(\XX', \LL')=0$.

Let $\pi:\XX \to X$  and $\pi'= \pi \circ\nu: \XX' \to X$ be the natural maps to the coarse moduli space. Since $\pi_*(\OO_\XX) \cong \pi'_*(\OO_{\XX'}) \cong \pi_*(\nu_*(\OO_{\XX'}))$ and by Proposition \ref{prop:properties_pi_coarse}(4) the functors $\pi_*$ and $\pi'_*$ are exact, it follows that $\nu_*\OO_{\XX'} \cong \OO_{\XX}$, so $H^0(\XX, \OO^*_{\XX}) = H^0(\XX', \OO^*_{\XX'})$.

Moreover $\RR(\XX')= \RR(\XX)$ as a $\kk$-algebra. Indeed, since only the line bundles pulled back from $\XX$ have non-trivial sections, to calculate the Cox Ring of $\XX'$ one can use the pullback of the line bundles and of the isomorphisms to define the Cox ring of $\XX$.
Note that the grading of $R(\XX')$ is obtained from the grading of $R(\XX)$ via the pullback $\nu^*$.

Hence, $\XX$ is an \MD-stack if and only if $\XX'$ is.
\end{proof} 

\begin{corollary}\label{theorem.line bundle root preserves MD}
Let $\XX$ be a stack and $\XX'$ be a stack obtained by root constructions with line bundles on $\XX$. 
Then $\XX$ is an \MD-stack if and only if $\XX'$ is an \MD-stack.
\end{corollary}

\begin{theorem} \label{theorem.MD quotient is preserved}
Let $\XX$ be an \MD-quotient stack. Then any stack $\XX'$ obtained by root constructions along simple normal crossing divisors and with line bundles from $\XX$ is again an \MD-quotient.
\end{theorem}

\begin{proof}
Let $\nu\colon\XX'\to\XX$ be the induced map and 
\[
\XX = \stackquot{\Spec \RR(\XX) \setminus V(J_{\irr}) }{ \Hom(\Pic(\XX), \kk^*)}
\]
be the \MD-quotient description of $\XX$. 
By the Lemmata~\ref{lem:root_keeps_coarse} and~\ref{lem:pullback_ample_lb}, the line bundle $\nu^*\LL$ 
is still ample and 
the ideal $J_{\irr}'= \sqrt{\genby{H^0(\XX', \nu^*\LL)}}$ in $\RR(\XX')$ gives the irrelevant ideal of $\XX'$. 
The desired description of $\XX'$ 
as an \MD-quotient follows directly from Proposition~\ref{prop:root_description_quot}, after taking into account the results about  
the Cox rings and the Picard groups of the stacks obtained as roots 
from $\XX$ given in Theorems \ref{theorem.divisor root preserves MD} and 
\ref{theorem.line bundle root preserves MD},
.
\end{proof} 

\section{Characterisation of Mori dream stacks}

In this section, we give a characterisation of \MD-stacks among smooth Deligne-Mumford stacks and 
study the conditions under which an \MD-stack can be obtained by roots from its canonical stack which turns out to be an \MD-stack, too. 
The representation of \MD-stacks as \MD-quotient stacks will be our main focus.

We start our analysis of the relations between \MD-stacks and their moduli spaces with the following result.

\begin{theorem} \label{thm.moduli space is MD}
Let $\XX$ be an \MD-stack and let $X$ be its coarse moduli scheme. Then $X$ is an \MD-space.
\end{theorem}

\begin{proof}
Let $\XX$ be an \MD-stack and $\pi\colon\XX\to X$ the map to its coarse moduli space. Since $\XX$ is a smooth Deligne-Mumford stack, 
its coarse moduli space $X$ is normal and has at most quotient singularities. 

We prove that $\Cl(X)$ is finitely generated. Since $\Cl(X)= \Cl(\Xcan) = \Pic(\Xcan)$, we have just to compare  $\Pic(\Xcan)$ and $\Pic(\XX)$. The Leray spectral sequence applied to the map $\pi_{can}\colon \XX \to \Xcan$ and to the sheaf $\OO^*_{\XX}$ gives the exact sequence:
\[ 0\to H^1(\Xcan, {\pi_{can}}_* \OO^*_{\XX}) \to H^1(\XX, \OO^*_\XX) \to H^0(\Xcan, R^1{\pi_{can}}_*\OO^*_\XX) \to \ldots  \]
and then the injective morphism $H^1(\Xcan, {\pi_{can}}_* \OO^*_{\XX})= H^1(\Xcan, \OO_{\Xcan}^*) = \Pic(\Xcan) \hookrightarrow  H^1(\XX, \OO^*_\XX) = \Pic(\XX)$.
Since $\XX$ is an \MD-stack, $\Pic(\XX)$ is finitely generated and the same is true for $\Cl(X)$.

From Proposition \ref{prop:properties_pi_coarse}(5), $\pi_*\OO_\XX \cong \OO_X$, so $H^0(X, \OO^*_X) = H^0(\XX, \OO^*_{\XX})=\kk^*$.

Let us prove that the Cox ring of $X$ is a finitely generated $\kk$-algebra. 
Since $H^0(\Xcan,\LL) = H^0(\XX,\pi_{can}^* \LL)$ and since the tensor product commutes with pullback, $\RR(X) = \RR(\Xcan)$ is a subalgebra of $\RR(\XX)$.

Note that by \cite[Lem.~3.2]{Kresch-Vistoli}, for any homogeneous $r \in \RR(\XX)$ there is a $d>0$ such that $r^d \in \RR(X)$. 
Actually there exists a $d>0$ which works for all homogeneous elements in $\RR(\XX)$. Choose a finite set of homogeneous generators $r_1,\ldots,r_m$ of $\RR(\XX)$.
Then $B = \{ \prod_i r_i^{j_i} \mid 1 \leq j_i < d\}$ generates $\RR(\XX)$ as an $\RR(X)$-module.
Indeed, write any element of $\RR(\XX)$ as a polynomial in the $r_i$ with coefficients in $\kk$. We can write any factor $r_i^a$ of a monomial as $r_i^{bd}\cdot r_i^c$ with $0 \leq c < d$. Note that $r_i^{bd}$ is already in $\RR(X)$. So the polynomial 
becomes a linear combination of elements in $B$ with coefficients in $\RR(X)$.

So we have a chain of inclusions $\kk \into \RR(X) \into \RR(\XX)$ with $\RR(\XX)$ finitely generated as a $\RR(X)$-module.
Hence we can apply \cite[Prop.~7.8]{Atiyah-Macdonald}, so $\RR(X)$ is finitely generaded as a $\kk$-algebra.
\end{proof}

For the characterisation of \MD-orbifolds the following general theorem plays a crucial role.

\begin{proposition}[{\cite[Thm.~6.1]{Geraschenko-Satriano}}] \label{prop.each orbifold is a root}
Let $\XX$ be a smooth Deligne-Mumford orbifold, whose ramification divisor has simple normal crossing components.
Then $\XX$ can be obtained by roots from its canonical stack $\Xcan$.
\end{proposition}

To avoid confusion, we want to note that \cite[Thm.~6.1]{Geraschenko-Satriano} is actually formulated for arbitrary ramification divisors. In such a case the root constructions with these divisors may yield a singular stack, which is then approximated with a smooth stack. This additional step is not necessary here.

\begin{lemma}[{\cite[Lem.~5.4]{Geraschenko-Satriano}}]
\label{lem.abelian-snc}
Let $U$ be a smooth algebraic space and $G$ a diagonalisable group scheme (e.g., subgroup of a torus) which acts properly on $U$ with finite stabilisers. 
Then the ramification divisor on $[U/G]$ has simple normal crossing components.
\end{lemma}

Theorems~\ref{theorem.divisor root preserves MD} and \ref{thm.moduli space is MD} and  Proposition~\ref{prop.each orbifold is a root} imply that smooth Deligne-Mumford orbifolds
whose ramification divisor has simple normal crossing components are \MD-orbifolds exactly when their coarse moduli spaces are \MD-spaces.
In combination with Example~\ref{ex:lb_roots_coarse} and Theorem~\ref{thm.gerbes preserve MD} we have the following result, which characterise \MD-stacks among the smooth Deligne-Mumford stacks.

\begin{corollary} \label{cor.all stacks with a MD moduli space are MD}
Let $\XX$ be a smooth Deligne-Mumford stack with abelian generic stabiliser, whose ramification divisor has simple normal crossing components. 
Then $\XX$ is an \MD-stack if and only if its coarse moduli space $X$ is an \MD-space. 
\end{corollary}

Our last aim is to understand when an \MD-stack can be obtained by roots from its canonical stack. 

\begin{theorem}
\label{thm:main_thm}
Let $\XX=\raisebox{0.5ex}{$\stackquot{\Spec \RR(\XX) \setminus V(J_{\irr}) }{  \Hom(\Pic(\XX), \kk^*) }$}$ be an \MD-quotient stack. 
Then it can be obtained by roots from its canonical stack $\Xcan$,
which is itself an \MD-quotient stack.
\end{theorem}

\begin{proof} 
Let $Z=\Spec \RR(\XX) \setminus V(J_{\irr})$ and $G=\Hom(\Pic(\XX), \kk^*)$.
Since $\XX$ is smooth, $Z$ is also smooth and
by Lemma \ref{lem.abelian-snc}, the ramification divisor of $\XX$ has simple normal crossing components.   

Let $\nu\colon  \XX \to \Xrig$ be the natural map from $\XX$ to its rigidification.
Since the rigidification preserves the ramification divisor, by Proposition \ref{prop.each orbifold is a root} the orbifold $\Xrig$ is obtained by roots of divisors from its canonical stack $\Xcan$. 

Let us prove that $\Xrig$ is indeed an \MD-quotient stack. 
Note that, by Theorem \ref{thm.gerbes preserve MD}, we have also $Z = \Spec \RR(\Xrig) \setminus V(J_{\irr})$. By Corollary \ref{cor:smooth_iff_snc}, the smoothness of $Z$ implies the smoothness of $\Spec \RR(\Xcan) \setminus V(J_{\irr})$ and then of $\Spec \RR(X)  \setminus V(J_{\irr})$, where $X$ is the coarse moduli space of these stacks. 
Therefore, $\Xcan$ is an \MD-quotient stack.

By Theorems \ref{theorem.canonical MD-stack} and \ref{theorem.MD quotient is preserved}, the stack $\Xrig$ is the \MD-quotient $\Xrig= [Z / \Grig]$, where we denote by $\Grig =\Hom(\Pic(\Xrig), \kk^*)$.

Hence we only need to show that $\XX$ is indeed obtained by root constructions with line bundles on $\Xrig$.
By construction, the map $\nu\colon \XX \to \Xrig$ is a $H$-gerbe, where $H$ is an abelian group scheme. 
As in~\eqref{Leray for gerbes}, 
the Leray spectral sequence for $\nu$ assures the existence of the following exact sequence
\[ 0 \to \Pic(\Xrig) \xrightarrow{\nu^*} \Pic(\XX) \to \Pic(\BB H) \to H^2(\Xrig, \mathbb{G}_m). \]
We can write $\coker \nu^* = \bigoplus^n_{i=1} \ZZ/ r_i\ZZ$, so we get the following short exact sequence
\[ 0 \to \Pic(\Xrig) \to \Pic(\XX) \to \bigoplus^n_{i=1} \ZZ/ r_i\ZZ \to 0 \] 
and, dualizing it, the short exact sequence:
\[ 0 \to \prod^n_{i=1} \mu_{r_i} \to G \to \Grig \to 0.\]
Choose now line bundles $\MM_i \in \Pic(\XX)$ which are mapped to a primitive generator in $\bar{1}\in \ZZ/r_i\ZZ$, and to $0$ otherwise.
Then $\MM^{r_i} = \LL_i$ for a line bundle $\LL_i \in \Pic(\Xrig)$.
By construction, $G$ fits in the following diagram with exact rows: 
\[
\xymatrix@R=3ex{
0 \ar[r]& H'\ar@{=}[d] \ar[r] & G \ar[r] \ar[d]_{\chi({\underline \MM})} & \Grig \ar[r] \ar[d]^{\chi({\underline \LL})} & 0\\
0 \ar[r] & H' \ar[r] & (\kk^*)^n \ar[r]_{\wedge \underline r} & (\kk^*)^n \ar[r] & 0,}
\]
where $H'=  \prod^n_{i=1} \mu_{r_i}$ and $\chi$ indicates the characters of line bundles. 
This implies that the right square in the diagram is cartesian. 

Recall that the $G$-action on $Z$ is given by the $\Pic(\XX)$-grading on $\RR(\XX)$. 
The proof of Theorem \ref{thm.gerbes preserve MD} shows that the grading of $\RR(\XX) = \RR(\Xrig)$ are 
related through the pullback 
$\nu^*\colon\Pic(\Xrig) \to \Pic(\XX)$, i.e., $G$ acts on $Z$ via the map $G \to \Grig$ 
given by the dual of the pullback.    

Thus, by Proposition~\ref{prop:root_description_quot}, $\XX=[Z/G]$ is canonically isomorphic to the stack 
$\sqrt[\underline r]{\underline \LL/\Xrig} $, where $\underline r=(r_1, \ldots , r_n)$ and 
$\underline \LL =(\LL_1, \ldots , \LL_n)$.
\end{proof}

\section{Mori dream stacks and smooth toric Deligne-Mumford stacks}

The following propositions are generalisations of \cite[Cor.~2.10]{Hu-Keel} and \cite[Prop.~2.11]{Hu-Keel}.
We want to remind the reader that a \emph{smooth toric Deligne-Mumford stack} $\XX$ is a smooth seperated Deligne-Mumford stack with an action of a Deligne-Mumford torus $\TT$, which is contained as an open dense orbit; see~\cite{Fantechi-etal}.
Here a Deligne-Mumford torus $\TT$ means that $\TT$ is isomorphic to $T\times \BB G$, where $T$ is an ordinary torus and $G$ a finite abelian group.

\begin{proposition} \label{thm. MDquotient Cox ring polyn is toric}
Let $\XX$ be an \MD-quotient stack.
Then $\XX$ is a smooth toric Deligne-Mumford stack if and only if its Cox ring is a polynomial ring.
\end{proposition}

\begin{proof}
First we assume that
 $\RR(\XX) = \kk[x_1,\ldots,x_n]$.
Therefore $T = (\kk^*)^n$ acts naturally on $\Spec \RR(\XX) = \AA^n$ by componentwise multiplication.
Fix the presentation $\ZZ^n \onto \Pic(\XX)$ where $e_i$ is sent to the divisor class $[D_i]$ given by $\{x_i = 0\}$.
By applying $\Hom(\blank,\kk^*)$ to this surjection,
we get an inclusion of $G = \Hom(\Pic(\XX),\kk^*)$ into $T$.
Note that the action of $G$ is given as a subgroup of $T$.

Furthermore, the locus $V(J_{\irr})$ is cut out by some of the hyperplanes $\{x_i = 0\}$,
since we can choose the ample line bundle to be given as $\OO(D)$ with 
$D = \sum a_i D_i$, so the global sections of $\OO(D)$ have a basis consisting of mononials in $x_i$.
Therefore, $V(J_{\irr})$ is $T$-invariant and consequently $T$ acts on $Z = \Spec \RR(\XX) \setminus V(J_{\irr})$.

This action descends to an action of $[T/G]$ on $[Z/G]$.
$[T/G]$ has again a torus $T'$ as its coarse moduli space (of dimension $n - \rk G$), 
so by \cite[Sec.~6.1]{Fantechi-etal} $[T/G]$ is a Deligne-Mumford torus, i.e.,  $[T/G] = T' \times \BB H$ for some finite abelian group $H$.
Moreover, since $T$ is open and dense in $Z$, the same holds true for $[T/G]$ inside $[Z/G]$.
Hence $[Z/G]$ is a smooth toric Deligne-Mumford stack.

The converse direction was shown in \cite[Thm.~7.7]{Fantechi-etal}.
\end{proof}

\begin{proposition} \label{thm. MDquotient embedded in toric}
Let $\XX$ be an \MD-quotient stack such that divisor class group of its coarse moduli space is free.
Then there is a closed embedding $\iota\colon\XX \into \YY$ with $\YY$ a smooth toric Deligne-Mumford stack such that $\iota^*\colon \Pic(\YY) \xrightarrow{\sim} \Pic(\XX)$ is an isomorphism.
\end{proposition}

\begin{proof}
First we note that $\XX$ can be obtained by roots from its canonical stack $\Xcan$, by Theorem~\ref{thm:main_thm}.

Let $X$ be the coarse moduli space of $\XX$. For the Cox ring $\RR(X)$ choose a presentation $\RR(X) = \kk[x_1,\ldots,x_m]/I$ where $I$ is a $\Cl(X)$-homogeneous ideal and where $\{x_i=0\}$ define divisors $D_i$ on $X$. Moreover, we assume that among these divisors $D_i$ 
there are those with which the root constructions are performed to obtain $\XX$ (after pulling them back to $\Xcan$).

Since $\Cl(X)$ is assumed to be free, by \cite[Prop.~5.2]{Berchtold-Hausen}, there is a toric variety $Y$ and a closed embedding $X\into Y$ inducing an isomorphism $\Cl(Y) \xrightarrow{\sim} \Cl(X)$ by pullback.
The inclusion $X \into Y$ is induced by the presentation $\RR(Y) = \kk[x_1,\ldots,x_m] \onto \RR(X)$.
The inclusion $X \into Y$ gives also an inclusion $\Xcan \into \Ycan$ with $\Pic(\Ycan) \xrightarrow\sim \Pic(\Xcan)$, by Theorems~\ref{theorem.canonical MD-stack} and \ref{thm:main_thm}.

Now we perform the root constructions on $\Xcan$ along the divisors to obtain $\Xrig$. By performing the analogous root constructions on $\Ycan$, we get a smooth toric Deligne-Mumford orbifold which we denote by $\Yrig$.
It is clear that the inclusion lifts to $\Xrig \into \Yrig$ and $\Pic(\Yrig) \xrightarrow\sim \Pic(\Xrig)$.

Finally, we perform the root constructions with line bundles to obtain $\XX$ from $\Xrig$.
If we do the same constructions starting with $\Yrig$, we arrive at the statement.
\end{proof}

\input{biblio-md-stacks}

\end{document}

%% file: header-md-stacks.tex
\usepackage{amsmath,amssymb,amsthm,amsfonts}
\usepackage{color}
\usepackage[all]{xy}
\SelectTips{cm}{}
\usepackage{tikz}
\usepackage{hyperref}
\usepackage{stmaryrd} 
\usepackage{enumitem}
\setitemize[0]{leftmargin=*}
\setenumerate[0]{leftmargin=*}

\definecolor{rosso}{RGB}{162,0,0}
\definecolor{verde}{RGB}{0,100,0}
\definecolor{blu}{RGB}{0,0,162}

\newcommand{\MD}{\textsf{MD}}

\newcommand{\CC}{\mathbb{C}}
\newcommand{\QQ}{\mathbb{Q}}
\newcommand{\ZZ}{\mathbb{Z}}

\renewcommand{\AA}{\mathbb{A}} 

\newcommand{\XX}{\mathcal{X}} 
\newcommand{\YY}{\mathcal{Y}} 

\newcommand{\can}{{\mathop{can}}}
\newcommand{\rig}{{\mathop{rig}}}
\newcommand{\Xcan}{\mathcal{X}^{\can}}
\newcommand{\Xrig}{\mathcal{X}^{\rig}}
\newcommand{\Ycan}{\mathcal{Y}^{\can}}
\newcommand{\Yrig}{\mathcal{Y}^{\rig}}

\newcommand{\DD}{\mathcal{D}} 
\newcommand{\FF}{\mathcal{F}}
\newcommand{\OO}{\mathcal{O}} 
\newcommand{\LL}{\mathcal{L}} 
\newcommand{\RR}{\mathcal{R}} 

\newcommand{\MM}{\mathcal{M}} 
\newcommand{\TT}{\mathcal{T}} 
\newcommand{\Grig}{G^{\rig}}

\newcommand{\quot}[2]{\left. \raisebox{0.5ex}{$\displaystyle #1$} \middle/ \raisebox{-0.5ex}{$\displaystyle #2$} \right.}
\newcommand{\stackquot}[2]{\left[ \raisebox{0.5ex}{$\displaystyle #1$} \middle/ \raisebox{-0.5ex}{$\displaystyle #2$} \right]}

\newcommand{\kk}{\mathbf{k}} 
\newcommand{\pt}{\mathsf{pt}} 

\newcommand{\BB}{\mathcal{B}} 
\newcommand{\Gm}{\mathbb{G}_m}
\newcommand{\thickslash}{\mathbin{\!\!\pmb{\fatslash}}} 

\DeclareMathOperator{\rk}{rk}

\DeclareMathOperator{\coker}{coker}
\DeclareMathOperator{\Hom}{Hom}

\DeclareMathOperator{\QCoh}{QCoh}

\DeclareMathOperator{\id}{id}
\let\div\relax
\DeclareMathOperator{\div}{div}

\DeclareMathOperator{\Pic}{Pic}

\DeclareMathOperator{\WDiv}{WDiv}
\DeclareMathOperator{\Cl}{Cl}
\DeclareMathOperator{\Spec}{Spec}

\DeclareMathOperator{\chara}{char}

\DeclareMathOperator{\Dic}{Dic} 

\DeclareMathOperator{\sm}{sm} 
\DeclareMathOperator{\irr}{irr} 

\newcommand{\genby}[1]{{\langle #1 \rangle}} 
\newcommand{\Genby}[1]{{\left\langle #1 \right\rangle}} 
\newcommand{\et}{\text{\'et}} 

\newcommand{\Het}{H_{\et}} 

\newcommand{\coloneqq}{\mathrel{\mathop:}=}

\newcommand{\into}{\hookrightarrow}
\newcommand{\onto}{\twoheadrightarrow}

\newcommand{\blank}{\,\text{--}\,} 
\newcommand{\git}{/\!\!/}

\newtheorem{theorem}{Theorem}[section]
\newtheorem*{maintheorem}{Main Theorem}
\newtheorem*{theorem*}{Theorem}
\newtheorem{corollary}[theorem]{Corollary}
\newtheorem{proposition}[theorem]{Proposition}
\newtheorem{lemma}[theorem]{Lemma}
\newtheorem*{lemma*}{Lemma}

\theoremstyle{definition}
\newtheorem{definition}[theorem]{Definition}
\newtheorem{remark}[theorem]{Remark}
\newtheorem{example}[theorem]{Example}
\newtheorem*{example*}{Example}

\author{Andreas Hochenegger}
\author{Elena Martinengo}

\newcommand{\bib}[6]{{\bibitem{#2} #3, {\emph{#4},} #5#6.}}
\newcommand{\arXiv}[1]{{\href{http://arxiv.org/abs/#1}{\texttt{arXiv:#1}}}}

%% file: biblio-md-stacks.tex
\addtocontents{toc}{\protect\setcounter{tocdepth}{-1}}

\bigskip
\noindent
\begin{tabular}{l p{0.5\textwidth}}
\emph{Email addresses:} & \verb+ahochene@math.uni-koeln.de+
                          \verb+martinengo@math.uni-hannover.de+
\end{tabular}